\documentclass[11pt]{amsart}
\usepackage{amssymb}
\usepackage{latexsym}
\usepackage{graphics}
\usepackage{mathptmx}
\usepackage{amsmath}
\usepackage{amsxtra}
\usepackage{amstext}
\usepackage{amssymb}
\usepackage{latexsym}
\usepackage{dsfont}
\usepackage{verbatim}

\DeclareMathOperator{\capa}{cap}

\newcommand{\defeq}{\mathrel{\mathop:}=}

\newcommand{\R}{{\mathbb R}}

\begin{document}
\title[Weighted norm inequalities of $(1,q)$-type] {Weighted norm inequalities of $(1,q)$-type for integral 
and fractional maximal operators}
\author{Stephen Quinn}
\address{Department of Mathematics, University of Missouri, Columbia, MO  65211}
\email{stephen.quinn@mail.missouri.edu}
\author{ Igor E. Verbitsky}
\address{Department of Mathematics, University of Missouri, Columbia, MO  65211}
\email{verbitskyi@missouri.edu}

\begin{abstract}
We study weighted norm inequalities of $(1,q)$- type for $0<q<1$, 
$$
\Vert \mathbf{G} \nu \Vert_{L^q(\Omega, d \sigma)} \le C \, \Vert \nu \Vert, \quad \text{for all positive measures $\nu$ in $\Omega$}, 
$$ 
along with their weak-type counterparts, where $\Vert \nu \Vert=\nu(\Omega)$,  and $G$ 
is an integral operator with nonnegative kernel, $$\mathbf{G} \nu(x) = \int_\Omega G(x, y) d \nu(y).$$ 

These problems are motivated by sublinear elliptic equations  
in a domain $\Omega\subset\R^n$ with non-trivial Green's function $G(x, y)$ associated with the Laplacian, fractional Laplacian, or more general elliptic operator.  

We also treat fractional maximal operators $M_\alpha$ ($0\le \alpha<n$) on $\R^n$, and  characterize strong- and weak-type  $(1,q)$-inequalities for $M_\alpha$ and more general maximal operators, 
 as well as $(1,q)$-Carleson measure inequalities for Poisson integrals.  
\end{abstract}

\dedicatory{Dedicated to Richard L. Wheeden}

\maketitle

\numberwithin{equation}{section}
\newtheorem{theorem}{Theorem}[section]
\newtheorem{lemma}[theorem]{Lemma}
\newtheorem{remark}[theorem]{Remark}
\newtheorem{cor}[theorem]{Corollary}
\newtheorem{prop}[theorem]{Proposition}
\newtheorem{defn}[theorem]{Definition}
\allowdisplaybreaks

\section{Introduction}

In this paper, we discuss recent results on weighted norm inequalities of $(1,q)$- type in the case $0<q<1$,   
\begin{equation}\label{strong-intr}
\Vert\mathbf{G} \nu\Vert_{L^q(\Omega, d \sigma)} \le C \, \Vert\nu\Vert, 
\end{equation} 
for all positive measures $\nu$ in $\Omega$, where $\Vert\nu\Vert=\nu(\Omega)$, 
and  $\mathbf{G}$ is an integral operator with nonnegative kernel, 
$$\mathbf{G} \nu(x) = \int_\Omega G(x, y) d \nu(y).$$

 Such problems are motivated by sublinear elliptic equations  
of the type 
\begin{equation*}\begin{cases}\label{introdir}-\Delta u = \sigma u^q \,\,  \text{ in } \Omega,\\
\; u = 0\text{ on }\partial\Omega,\end{cases}\end{equation*} 
in the case $0<q<1$, where $\Omega$ is an open set in $\R^n$ with non-trivial Green's function $G(x, y)$,   and $\sigma \ge 0$ is an arbitrary locally integrable function, or locally finite measure in $\Omega$.  

The only restrictions imposed on the kernel $G$ are that it is quasi-symmetric and satisfies a weak maximum principle. In particular, $\mathbf{G}$ can be a Green operator associated with the Laplacian, a more general elliptic operator  (including the fractional Laplacian), or a convolution operator on $\R^n$ with radially symmetric 
decreasing kernel $G(x, y)= k(|x-y|)$ (see \cite{AH}, \cite{L}). 

In particular, we consider in detail the one-dimensional case where $\Omega=\R_{+}$ and $G(x, y) = \min(x, y)$. We deduce explicit characterizations of the corresponding $(1, q)$-weighted norm inequalities, give 
explicit necessary and sufficient conditions for the existence of weak solutions, and obtain sharp 
two-sided pointwise estimates of solutions.

We also characterize weak-type counterparts of \eqref{strong-intr}, namely, 
\begin{equation}\label{weak-intr}
\Vert\mathbf{G} \nu\Vert_{L^{q, \infty}(\Omega, d \sigma)} \le C \, \Vert\nu\Vert. 
\end{equation} 

Along with integral operators, we treat fractional maximal operators $M_\alpha$ with $0\le \alpha<n$ 
on $\R^n$, and  characterize both strong- and weak-type  $(1,q)$-inequalities for $M_\alpha$, 
and more general maximal operators. Similar problems for Riesz potentials were studied earlier 
in \cite{CV1}--\cite{CV3}. Finally, we apply our results for the integral operators to the Poisson kernel to characterize a $(1,q)$-Carleson measure inequality.

\section{Integral Operators}
\subsection{Strong-Type $(1,q)$-Inequality for Integral Operators} Let $\Omega \subseteq \R^n$ be a connected open set. By $\mathcal{M}^+(\Omega)$ we denote 
	the class of all nonnegative locally finite Borel measures in $\Omega$. Let $G\colon \Omega \times \Omega \to [0, +\infty]$ be a nonnegative lower-semicontinuous kernel. We will assume throughout this paper that $G$ is quasi-symmetric, i.e., there exists a constant $a>0$ such that 
	\begin{equation}
		\label{q-symm}
	a^{-1} \, G(x, y) \le G(y, x) \le a \, G(x, y), \quad x, y \in \Omega.
\end{equation}

If $\nu  \in \mathcal{M}^+(\Omega)$, then by $\mathbf{G} \nu$ and $\mathbf{G}^{*} \nu$ we denote the integral operators  (potentials) defined respectively by 
\begin{equation}
		\label{pot}
		\mathbf{G} \nu(x) = \int_{\Omega} G(x, y) \, d \nu(y), \quad \mathbf{G}^{*}\nu(x) = \int_{\Omega} G(y, x) \, d \nu(y), \quad x \in \Omega. 
		\end{equation}

We say that the kernel $G$ satisfies the \textit{weak maximum principle} if, for any constant $M>0$,  the inequality 
		\[ \mathbf{G}\nu(x) \le M \quad \text{for all} \, \,  \, x \in S(\nu) \]
		 implies
		\[ \mathbf{G}\nu(x) \le h M \quad \text{for all} \, \, \,   x \in \Omega, \] 
		where $h\ge 1$ is a constant, and $S(\nu) \defeq {\rm supp} \, \nu$. 
	When $h = 1$, we say that $\mathbf{G}\nu$ satisfies the \textit{strong maximum principle}. 
	
	It is well-known that Green's kernels associated with many partial differential operators are quasi-symmetric, and satisfy the weak maximum principle (see, e.g., \cite{An1}, \cite{An2}, \cite{L}).

The kernel $G$ is said to be \textit{degenerate} with respect to $\sigma \in \mathcal{M}^+(\Omega)$ provided there exists a set $A \subset \Omega$ with $\sigma(A) > 0$ and
		\[ G(\cdot, y) = 0 \quad \text{$d\sigma$-a.e. for $y \in A$.} \] 
		Otherwise, we will say that $G$ is \textit{non-degenerate} with respect to $\sigma$. (This notion was 
		introduced in \cite{Sinnamon} in the context of $(p, q)$-inequalities for positive operators 
		$T\colon L^p \to L^q$ in the case $1<q<p$.) 

Let $0<q<1$, and let $G$ be a kernel on $\Omega\times\Omega$. For $\sigma \in  \mathcal{M}^+(\Omega)$, 
we consider the problem of the existence of a \textit{positive solution} $u$ 
to the integral equation 
\begin{equation}
		\label{int-eq}
		u= \mathbf{G} (u^q d \sigma) \quad \text{in} \,\, \, \Omega, \quad 0<u<+\infty \,\, \, d\sigma{\rm -a.e.}, \quad  u\in L^q_{{\rm loc}} (\Omega). 
		\end{equation}
		We call $u$ a positive \textit{supersolution} if  
		\begin{equation}
		\label{int-sup}
		u\ge  \mathbf{G} (u^q d \sigma) \quad \text{in} \,\, \, \Omega, \quad 0<u<+\infty \,\,\, d\sigma{\rm -a.e.}, \quad  u\in L^q_{{\rm loc}} (\Omega). 
\end{equation}
		
		This is a generalization of the sublinear elliptic problem (see, e.g., \cite{BK}, \cite{BO},  and the literature cited there): 
		\begin{equation}\label{pde} 
		\begin{cases} 
		-\Delta u = \sigma u^q \quad \text{in}  \, \, \Omega,\\
\; u = 0 \quad \text{on} \, \, \partial\Omega, \end{cases}
\end{equation} 
where $\sigma$ is a nonnegative locally integrable function, or measure, in $\Omega$.

If $\Omega$ is a 
 bounded $C^2$-domain then solutions to \eqref{pde}  can be understood in the ``very weak'' sense (see, e.g.,  \cite{MV}). For general 
domains $\Omega$ with a nontrivial  Green function $G$ associated with the Dirichlet Laplacian $\Delta$ in $\Omega$, 
solutions $u$ are understood as in \eqref{int-eq}.

\begin{remark} In this paper, for the sake of simplicity, we sometimes consider positive solutions and supersolutions $u \in L^q(\Omega, d \sigma)$. In other words, we replace the natural local condition $u\in L^q_{{\rm loc}} (\Omega, d \sigma)$ with its global counterpart. Notice that the local condition is necessary for solutions (or supersolutions) 
to be properly defined.
\end{remark}

To pass from solutions $u$ which are globally in $L^q(\Omega, d \sigma)$ to  
all solutions $u\in L^q_{{\rm loc}} (\Omega, d \sigma)$ (for instance, very weak solutions to \eqref{pde}), 
one can use either  a localization method developed in \cite{CV2} (in the case of Riesz kernels on $\R^n$), or 
\textit{modified} kernels $\widetilde G(x, y) = \frac{G(x, y)}{m(x) \, m(y)}$, where the modifier 
$m(x) = \min \Big(1, G(x, x_0)\Big)$ (with a fixed pole $x_0 \in \Omega$)  plays the role of a regularized 
distance to the boundary $\partial \Omega$. One also needs to consider the corresponding $(1,q)$-inequalities with a weight  $m$ (see \cite{QV}). See the next section in the one-dimensional case 
where $\Omega=(0, +\infty)$. 

\begin{remark}  Finite energy solutions, for instance, solutions $u \in W^{1,2}_{0} (\Omega)$ to \eqref{pde},  require the global condition $u \in L^{1+q}(\Omega, d \sigma)$, and are easier to characterize (see \cite{CV1}, \cite{QV}).
\end{remark}

The following theorem is proved  in \cite{QV}. (The case where $\Omega=\R^n$ and $\mathbf{G}=(-\Delta)^{-\frac{\alpha}{2}}$ is the Riesz potential of order $\alpha\in (0, n)$ was considered earlier in \cite{CV2}.)
	
\begin{theorem}\label{strong-thm}
		Let $\sigma \in \mathcal{M}^{+}(\Omega)$, and $0<q<1$.  Suppose $G$ is a quasi-symmetric kernel which satisfies the weak maximum principle.
		Then the following statements are equivalent:
		\begin{enumerate}
			\item There exists a positive constant $\varkappa=\varkappa(\sigma)$ such that 
			$$\Vert  \mathbf{G}\nu \Vert_{L^{q}(\sigma)} \le \varkappa \Vert \nu \Vert \quad \text{\rm for all} \,\,   \nu \in \mathcal{M}^{+}(\Omega).$$ 
			\item There exists a positive supersolution $u \in L^q(\Omega, d \sigma)$ to  \eqref{int-sup}. 
			\item There exists a positive solution $u\in L^q(\Omega, d \sigma)$ to  \eqref{int-eq}, provided 
			additionally that $G$ is non-degenerate with respect to $\sigma$.  
		\end{enumerate}
	\end{theorem}
	
	\begin{remark} The implication {\rm(1) $\Rightarrow$ (2)} in Theorem \ref{strong-thm} holds for any nonnegative kernel $G$, without assuming 
	that it is either quasi-symmetric, or satisfies the weak maximum principle.  This is a consequence 
	of Gagliardo's lemma \cite{G}; see details in \cite{QV}. 
	\end{remark}
	
	\begin{remark} The implication {\rm (3) $\Rightarrow$ (1)} generally fails  for kernels $G$ which do not satisfy  the weak maximum principle (see examples in \cite{QV}).  \end{remark}
	
	The following corollary of Theorem \ref{strong-thm} is obtained in \cite{QV}. 
	\begin{cor}\label{cor} Under the assumptions of Theorem \ref{strong-thm}, if there exists a positive supersolution $u \in L^q(\Omega, \sigma)$ to  \eqref{int-sup}, then $\mathbf{G}\sigma \in L^{\frac{q}{1-q}} (\Omega, d \sigma)$. 
	
	Conversely, if $\mathbf{G}\sigma \in L^{\frac{q}{1-q}, 1} (\Omega, d \sigma)$, then there exists 
	a non-trivial supersolution $u \in L^q(\Omega, \sigma)$ to \eqref{int-sup}  (respectively,  a solution $u$, provided $G$ is non-degenerate with respect to $\sigma$). 
\end{cor} 
	
	\subsection{The One-Dimensional Case} In this section, we consider positive weak solutions  
	  to sublinear ODEs of the type \eqref{pde} on the semi-axis $\R_{+}=(0, +\infty)$. It is instructive to consider the one-dimensional case where elementary characterizations 
	  of $(1, q)$-weighed norm inequalities, along with the corresponding existence theorems and explicit global pointwise estimates of  solutions are available. Similar results hold for sublinear equations on any interval 
	$(a, b) \subset \R$.

	Let $0<q<1$, and let $\sigma \in \mathcal{M}^{+} (\R_{+})$. Suppose $u$ is a positive weak solution to the equation 
			 \begin{equation}\label{one-dim}
		-u'' = \sigma u^q \quad \text{on} \,\, \,  \R_{+}, \quad u(0)=0, 
			\end{equation}
such that $\lim_{x \to +\infty}\frac{u(x)}{x}=0$. This condition at infinity ensures 
that $u$ does not contain a linear component. Notice that we assume that $u$ is concave and increasing on $[0, +\infty)$, and $\lim_{x \to 0^{+}} u(x)=0$.

	In terms of integral equations, we have $\Omega=\R_{+}$, and $G(x, y)=\min(x, y)$ is the Green function associated 
	with the Sturm-Liouville operator $\Delta u = u'' $ with zero boundary condition at $x=0$. Thus,  
	\eqref{one-dim} is equivalent to the equation 
\begin{equation}
		\label{one-dim-green}
		u(x) = \mathbf{G}(u^q d \sigma)(x) \defeq\int_0^{+\infty} \min(x, y) u(y)^q d \sigma (y), \quad x>0, 
			\end{equation}
			where $\sigma$ is a locally finite measure on $\R_{+}$, and 
			\begin{equation}\label{loc-cond} 
				\int_0^a y \, u(y)^q d \sigma(y) < +\infty, \quad \int_a^{+\infty} u(y)^q d \sigma (y)<+\infty, \quad  
				\text{{\rm for every}} \,\, a>0.
					\end{equation}
				 This ``local integrability'' condition ensures that the right-hand side of \eqref{one-dim-green} is well defined. Here intervals $(a, +\infty)$ are used in place of
				 balls $B(x, r)$ in $\R^n$.
				 
				  Notice that 
 \begin{equation}
\label{one-dir}
		u'(x) =  \int_x^{+\infty}  u(y)^q d \sigma(y), \quad x>0.
			\end{equation}
			Hence, $u$ satisfies the global integrability condition 
			\begin{equation}\label{glob-cond} 
				\int_0^{+\infty} u(y)^q d \sigma(y) < +\infty 
				\end{equation}
				if and only if $u'(0)< +\infty$. 
				
				The corresponding $(1,q)$-weighted norm inequality is given by 
				\begin{equation}
\label{one-weight}
				\Vert  \mathbf{G}\nu \Vert_{L^{q}(\sigma)} \le \varkappa \Vert \nu \Vert,
				\end{equation}
				where $\varkappa=\varkappa(\sigma)$ is a positive constant which does not depend on $\nu \in \mathcal{M}^+(\R_{+})$. Obviously,  \eqref{one-weight} is equivalent to 
				\begin{equation}
\label{hardy}
				\Vert H_{+} \nu + H_{-} \nu\Vert_{L^q(\sigma)} \le \varkappa \Vert \nu \Vert \quad \text{\rm for all} \,\,   \nu \in \mathcal{M}^+(\R_{+}),
				\end{equation}
				where $H_{\pm}$ is a pair of Hardy operators, 
				\begin{equation*}
				H_{+} \nu(x) = \int_0^x y \, d \nu(y), \quad H_{-} \nu(x) = x \int_x^{+\infty}  d \nu(y). 
		\end{equation*}
		
		The following proposition can be deduced from the known results on two-weight  Hardy inequalities 
		in the case $p=1$ and $0<q<1$ (see, e.g., \cite{SS}). We give here a simple  independent proof. 
		
		\begin{prop}\label{one-thm}
		Let $\sigma \in \mathcal{M}^{+}(\R_{+})$, and let $0<q<1$.  
		Then \eqref{one-weight} holds 
		if and only if 
				\begin{equation}
		\label{hardy-cond}
				\varkappa(\sigma)^q = \int_0^{+\infty} x^q d \sigma(x)< +\infty,
				\end{equation}
				where $\varkappa(\sigma)$ is the best constant in \eqref{one-weight}. 	\end{prop}
	\begin{proof}
Clearly,  
\begin{equation*}
				H_{+} \nu(x) + H_{-} \nu(x) \le x \,  \Vert\nu\Vert, \quad x>0.  
		\end{equation*}
		 Hence, 
		\begin{equation*}
					\Vert  H_{+}\nu +H_{-} \nu\Vert_{L^{q}(\sigma)} \le \Big(  \int_0^{+\infty}  x^q  d \sigma(x) \Big)^{\frac{1}{q}} \Vert\nu\Vert, 
		\end{equation*} 
		which proves \eqref{hardy}, and hence \eqref{one-weight},  with $\varkappa = \Big(  \int_0^{+\infty}  x^q  d \sigma(x) \Big)^{\frac{1}{q}}$.

		Conversely, suppose that \eqref{hardy} holds. 
	Then, for every $a>0$, and $\nu \in \mathcal{M}^+(\R_{+})$, 
		\begin{align*}
		& \Big( \int_{0}^{a} x^q  d \sigma(x)\Big) \Big( \int_a^{+\infty} d \nu(y)\Big)^q\\ &
				\le  \int_{0}^{a} \Big(x  \int_x^{+\infty} d \nu(y)\Big)^q d \sigma(x) 
\\& \le \int_{0}^{+\infty} (H_{-} \nu)^q  d \sigma \le \varkappa^q \Vert\nu\Vert^q. 
				\end{align*}
For $\nu=\delta_{x_0}$ with $x_0>a$, we get 
	\begin{equation*}
\int_{0}^{a} x^q  d \sigma(x) \le \varkappa^q. 
\end{equation*} 
Letting $a \to +\infty$, we deduce \eqref{hardy-cond}. \end{proof} 
			
			Clearly, the Green kernel $G(x,y)=\min(x, y)$ is symmetric, and satisfies the strong maximum 
			principle. Hence, by Theorem \ref{strong-thm}, equations \eqref{one-dim} and \eqref{one-dim-green} 
			have a non-trivial (super)solution $u\in L^q(\R_{+}, \sigma)$ if and only if \eqref{hardy-cond} 
			holds. 
			
			From Proposition~\ref{one-thm}, we deduce  that, for  ``localized'' measures  
			$d \sigma_a = \chi_{(a, +\infty)} d \sigma$ ($a>0$), we have 
			\begin{equation}\label{restr}
\varkappa(\sigma_a) = \Big(\int_{a}^{+\infty} x^q  d \sigma(x) \Big)^{\frac{1}{q}}. 
\end{equation} 

Using this observation and the localization method developed in \cite{CV2},  we obtain the following existence theorem for general weak solutions to \eqref{pde}, 
along with sharp pointwise estimates of solutions. 

We introduce a new potential 
	\begin{equation}
		\label{nonlin}
			\mathbf{K} \sigma (x) \defeq 
			x 	\Big ( \int_{x}^{+\infty} y^q  d \sigma(y)\Big)^{\frac{1}{1-q}}, \quad x>0.  			\end{equation}
			We observe that $\mathbf{K} \sigma $ is a one-dimensional analogue of the potential introduced 
			recently in \cite{CV2} in the framework of  intrinsic Wolff potentials in $\R^n$ (see also \cite{CV3} in the radial case). Matching upper and 
			lower pointwise bounds  of solutions are  obtained below by combining $\mathbf{G} \sigma$ 
			with $\mathbf{K} \sigma$.

	\begin{theorem}\label{one-pde-eq}
		Let $\sigma \in \mathcal{M}^{+}(\R_{+})$, and let $0<q<1$.  
		Then equation \eqref{pde}, or equivalently \eqref{one-dim} has a nontrivial solution 
		if and only if, for every $a>0$,  
				\begin{equation}
		\label{weak-cond}
				\int_0^a  x \, d \sigma(x) + \int_a^{+\infty} x^q d \sigma(x)< +\infty.
				\end{equation}
				Moreover, if \eqref{weak-cond} holds, then there exists a positive solution $u$ to 
				\eqref{pde} such that 
						\begin{align}
				\label{pointwise}
				C^{-1} &  \Big [ \Big(\int_0^x y \, d \sigma(y)\Big)^{\frac{1}{1-q}}  +  \mathbf{K} \sigma (x)\Big]
				\\ & \le u(x) \le C \, \Big [ \Big(\int_0^x y \, d \sigma(y)\Big)^{\frac{1}{1-q}}  +  \mathbf{K} \sigma (x)\Big].
								\end{align}
			The lower bound in \eqref{pointwise} holds for any non-trivial supersolution $u$. 	\end{theorem}
			
		\begin{remark}
		The lower bound  
		\begin{equation}\label{lower-1} 
		u(x) \ge (1-q)^{\frac{1}{1-q}} \Big [ \mathbf{G} \sigma(x)\Big]^{\frac{1}{1-q}}, \quad x>0, 
			\end{equation}
	is known for a general kernel $G$ which satisfies the strong maximum principle (see \cite{GV}, Theorem 3.3; \cite{QV}), and the constant $(1-q)^{\frac{1}{1-q}}$ here is sharp. However, the second term on the left-hand side of  \eqref{pointwise} makes the lower estimate stronger, so that it matches the upper estimate. 
	\end{remark}

	\begin{proof} The lower bound 
	\begin{equation}\label{est-1} 
		u(x) \ge (1-q)^{\frac{1}{1-q}} \Big[\int_0^x y \, d \sigma(y)\Big]^{\frac{1}{1-q}} , \quad x>0, 
			\end{equation}
			is immediate from \eqref{lower-1}. 
			
		Applying  Lemma 4.2 in \cite{CV2}, with the interval $(a, +\infty)$ in place of a ball $B$, and combining it with \eqref{restr}, for any $a>0$ we have 
		\begin{equation*}	
			\int_a^{+\infty} u(y)^q d \sigma(y) \ge c(q) \varkappa(\sigma_a)^{\frac{q}{1-q}} = c(q) \Big[	\int_a^{+\infty}  y^q d \sigma(y)\Big]^{\frac{1}{1-q}}. 
				\end{equation*}
			
			Hence,
			\begin{equation*}
		u(x) \ge \mathbf{G} (u^q d \sigma) \ge x \int_x^{+\infty} u(y)^q d \sigma(y) \ge c(q) \, x 
		\Big[	\int_x^{+\infty}  y^q d \sigma(y)\Big]^{\frac{1}{1-q}}. 
					\end{equation*}
			
			Combining the preceding estimate with \eqref{est-1}, we obtain the lower bound in \eqref{pointwise} 
			for any non-trivial supersolution $u$. This also proves that \eqref{weak-cond} is necessary 
			for the existence of a non-trivial positive supersolution.

			Conversely, suppose that  \eqref{weak-cond} holds. Let 
			\begin{equation}\label{super-sol} 
		v(x) \defeq c \,  \Big [ \Big(\int_0^x y \, d \sigma(y)\Big)^{\frac{1}{1-q}}  +  \mathbf{K} \sigma (x)\Big], \quad x>0, 
			\end{equation}
			where $c$ is a positive constant. It is not difficult to see that $v$ is a supersolution, so that $v \ge \mathbf{G} (v^q d \sigma)$, 
			if $c=c(q)$ is picked large enough. (See a similar argument in the proof of Theorem 5.1 in \cite{CV3}.)

			Also, it is easy to see that $v_0=c_0 (\mathbf{G} \sigma)^{\frac{1}{1-q}}$ is a subsolution, i.e.,  $v_0 \le \mathbf{G} (v_0^q d \sigma)$, 
			provided $c_0>0$ is a small enough constant. Moreover, we can ensure that  $v_0 \le v$ 
			if $c_0=c_0(q)$ is picked sufficiently small. (See details in \cite{CV3} in the case 
			of radially symmetric solutions in $\R^n$.) Hence, there exists a 
			solution which can be constructed by iterations,   starting from $u_0=v_0$, 		and 
			letting 
		\begin{equation*}	
u_{j+1} = \mathbf{G} (u_j^q d \sigma), \quad j = 0, 1, \ldots . 
	\end{equation*}
Then by induction $u_{j} \le u_{j+1} \le v$, and consequently $u = \lim_{j \to +\infty} u_j$ is a solution to \eqref{one-dim-green} by the Monotone Convergence Theorem. Clearly, $u \le v$, which proves the upper bound in \eqref{pointwise}. 
\end{proof}
		
\subsection{Weak-Type $(1,q)$-Inequality for Integral Operators}
	In this section, we characterize weak-type analogues of $(1,q)$-weighted norm inequalities considered above. We will use some elements of potential theory for general positive kernels $G$, including the notion of \textit{inner capacity}, $\capa(\cdot)$, and the associated \textit{equilibrium} (extremal) measure (see \cite{F}). 
	
	\begin{theorem}\label{weak-thm}
		Let $\sigma \in \mathcal{M}^{+}(\R^n)$, $0<q<1$, and $0\le \alpha <n$.  Suppose $G$ satisfies the weak maximum principle.
		Then the following statements are equivalent:
		\begin{enumerate}
			\item There exists a positive constant $\varkappa_w$ such that $$\Vert  \mathbf{G}\nu \Vert_{L^{q,\infty}(\sigma)} \le \varkappa_w \Vert \nu \Vert \quad \text{\rm for all} \,\,   \nu \in \mathcal{M}^+(\R^n).$$ 
			\item There exists a positive constant $c$ such that 
			$$\sigma(K) \le c \Big (\capa (K)\Big)^q \quad \text{\rm for all compact sets} \,\, K \subset \R^n.$$ 
			\item $\mathbf{G} \sigma \in L^{\frac{q}{1-q}, \infty}(\sigma)$.
		\end{enumerate}
	\end{theorem}
	\begin{proof}
		(1) $\Rightarrow$ (2)  Without loss of generality we may assume that the kernel $G$ is \textit{strictly positive}, that is, $G(x, x)>0$ for all $x \in \Omega$. Otherwise, 
		we can consider the kernel $G$ on the set $\Omega\setminus A$, where $A\defeq \{x \in \Omega\colon 
		G(x, x) \not=0\}$, since $A$ is negligible for the corresponding $(1,q)$-inequality in statement (1). (See details in \cite{QV} in the case of the corresponding strong-type inequalities.) 
		
		We remark that the kernel $G$ is known to be strictly positive if and only if, for any compact set $K \subset \Omega$, 
		the inner capacity ${\rm cap}(K)$ is finite (\cite{F}). In this case there exists an equilibrium measure $\lambda$  on $K$ such that 
		\begin{equation}
		\label{cap}
		\mathbf{G} \lambda \ge 1 \, \, \text{{\rm n.e.}} \,\, \text{{\rm on}}\,\, K, \quad \mathbf{G} \lambda \le 1 \, \,  \text{{\rm on}} \, \,  S(\lambda), \quad 
		\Vert\lambda\Vert={\rm cap}(K).
			\end{equation}
		Here n.e. stands for \textit{nearly everywhere}, which  means that the inequality holds on a given set except for a subset 
		of zero capacity \cite{F}. 
		
		Next, we remark that condition (1) yields that $\sigma$ is absolutely continuous with respect to capacity, i.e., $\sigma(K)=0$ if ${\rm cap}(K)=0$. 
		 (See a similar argument in \cite{QV} in the case of strong-type inequalities.) Consequently, 
		 $\mathbf{G} \lambda \ge 1$ $d \sigma$-a.e. on $K$. 
		 Hence,  by applying condition (1) with $\nu = \lambda$, we obtain (2).
		
		(2) $\Rightarrow$ (3) We denote by $\sigma_E$ denotes the restriction of $\sigma$ to a Borel set $E \subset \Omega$. 
		Without loss of generality we may assume that $\sigma$ is a finite measure on $\Omega$. Otherwise 
		we can replace $\sigma$ with $\sigma_F$ where  $F$ is a compact subset of $\Omega$. 
		We then deduce the estimate
		$$
		\Vert \mathbf{G} \sigma_F \Vert_{L^{\frac{q}{1-q}, \infty}(\sigma_F)} \le C< \infty,
		$$
		where $C$ does not depend on $F$, and use the exhaustion of $\Omega$ by an increasing sequence of compact subsets $F_n \uparrow \Omega$ to conclude that $ \mathbf{G}\sigma \in L^{\frac{q}{1-q}, \infty}(\sigma)$ 
		by the  Monotone Convergence Theorem.

		Set $E_t \defeq \{ x \in \Omega\colon \, \mathbf{G}\sigma (x)> t \}$, where $t>0$. 			Notice that, for all 
		$x \in (E_t)^c$,  
		$$
		\mathbf{G} \sigma_{(E_t)^c} (x) \le \mathbf{G} \sigma (x) \le t. 
		$$
		The set $(E_t)^c$ is closed, and hence  the preceding inequality holds on $S( \sigma_{(E_t)^c})$. It follows  by the weak maximum principle that, for all $x \in \Omega$, 
		$$
		\mathbf{G} \sigma_{(E_t)^c} (x) \le \mathbf{G} \sigma (x) \le h \,  t. 
		$$
	Hence, 
		\begin{equation}	\label{eq-h}
		\{x \in \Omega \colon \mathbf{G} \sigma (x) > (h+1) t 
		\} \subset \{x \in \Omega \colon \mathbf{G} \sigma_{E_t} (x) > t\}. 
		\end{equation} 
		
		Denote by $K\subset \Omega$ a compact subset of $\{x \in \Omega\colon \mathbf{G} \sigma_{E_t} (x) > t\}$. By (2), we have 
			\[ \sigma (K) \le c \, \Big(\capa (K)\Big)^q \]
		If $\lambda$ is the equilibrium measure on $K$, then $\mathbf{G} \lambda \le 1$ on $S(\lambda)$, and 
		$\lambda(K) = \capa(K)$  
		by \eqref{cap}. 
		Hence by the weak maximum principle $\mathbf{G} \lambda \le h$ on $\Omega$. 
		Using quasi-symmetry of the kernel $G$ and Fubini's theorem, we have
		\begin{align*}
			\capa (K) 
				&= \int_{K} \, d\lambda \\
				&\le \frac{1}{t} \int_{K} \mathbf{G} \sigma_{E_t} \, d\lambda \\
				&\le \frac{a}{t} \int_{E_t} \mathbf{G}\lambda d \sigma\\
				&\le \frac{a h}{t} \sigma(E_t).
		\end{align*}
		This shows that
		\begin{equation*}
			\sigma(K) \le \frac{c (a h)^q} {t^q} \, \Big (\sigma(E_t)\Big)^q.
		\end{equation*}
		Taking the supremum over all $K \subset E_t$, we deduce 
		\begin{equation*}
		\Big(\sigma(E_t) \Big)^{1-q} \le \frac{c (a h)^{q}} {t^q}. 
		\end{equation*}
		It follows from \eqref{eq-h} that, for all $t>0$,  
			\begin{equation*}
			t^{\frac{q}{1-q}} \sigma \Big ( \Omega \colon \mathbf{G} \sigma (x) > (h+1) t  \Big) \le 
			t^{\frac{q}{1-q}} \sigma(E_t) \le 
			c^{\frac{1}{1-q}} (a h)^{\frac{q}{1-q}}. 
				\end{equation*}
		Thus, (3) holds.
		
		(3) $\Rightarrow$ (2) 
			By H\"older's inequality for weak $L^q$ spaces, we have
			\begin{align*}
				\Vert \mathbf{G}\nu \Vert_{L^{q,\infty}(\sigma)}
					&= \left\Vert \frac{\mathbf{G}\nu}{\mathbf{G}\sigma} \mathbf{G}\sigma \right\Vert_{L^{q, \infty}(\sigma)} \\
					&\le \left\Vert \frac{\mathbf{G}\nu}{\mathbf{G}\sigma}  \right\Vert_{L^{1, \infty}(\sigma)} \left\Vert \mathbf{G}\sigma \right\Vert_{L^{\frac{q}{1-q}, \infty}(\sigma)} \\
					&\le C \left\Vert \mathbf{G}\sigma \right\Vert_{L^{\frac{q}{1-q}, \infty}(\sigma)} \Vert \nu \Vert,
			\end{align*}
			where the final inequality, 		$$
		\left\Vert \frac{\mathbf{G}\nu}{\mathbf{G}\sigma}  \right\Vert_{L^{1, \infty}(\sigma)} \le C \, 
		\Vert \nu \Vert, 
		$$	
			with a constant $C=C(h, a)$, was obtained  in \cite{QV}, for quasi-symmetric kernels $G$ satisfying the weak maximum principle.  
\end{proof}

\section{Fractional Maximal Operators}
	We denote by  $\mathcal{M}^+(\R^n)$ the class of positive locally finite Borel measures on $\R^n$. For $\nu \in \mathcal{M}^+(\R^n)$, we set $\Vert\nu\Vert=\nu(\R^n)$. 
	
	Let $\nu \in \mathcal{M}^+(\R^n)$, and let $0 \le \alpha < n$. 
	We define the fractional maximal operator $M_\alpha$ by
		\begin{equation}
		\label{frac-max}
 M_\alpha \nu (x) \defeq \sup_{Q \ni x} \frac{|Q|_\nu}{|Q|^{1-\frac{\alpha}{n}}}, \quad x \in \R^n,
 	\end{equation}
		where $Q$ is a cube, $|Q|_\nu \defeq \nu(Q)$, and $|Q|$ is the Lebesgue measure of $Q$. 
		If $f \in L^1_{\rm{loc}} (\R^n, d \mu)$ where $\mu \in \mathcal{M}^+(\R^n)$, 
		we set $M_\alpha (f d \mu)= M_\alpha \nu$ where $d \nu=|f| d\mu$, i.e., 
		\begin{equation}
		\label{frac-max-mu}
 M_\alpha (f d \mu) (x) \defeq \sup_{Q \ni x} \frac{1}{|Q|^{1-\frac{\alpha}{n}}} \int_Q |f| \, d\mu, \quad x \in \R^n. 
 	\end{equation}
		
		For $\sigma \in  \mathcal{M}^+(\R^n)$, it was shown in \cite{V1}  
		that in the case   $0<q<p$,  
				\begin{equation}
		\label{m-frac}
		M_\alpha\colon L^p(dx)\to L^q(d \sigma) \Longleftrightarrow M_\alpha \sigma \in L^{\frac{q}{p-q}} ( d \sigma),  
		 	\end{equation}
			\begin{equation}
		\label{m-frac-weak}
		M_\alpha\colon L^p(dx)\to L^{q, \infty} (d \sigma) \Longleftrightarrow M_\alpha \sigma \in L^{\frac{q}{p-q}, \infty} ( d \sigma),  
		 	\end{equation}
provided $p>1$.

	 More general two-weight maximal inequalities  
	 \begin{equation}
		\label{two-weight}
			\Vert M_\alpha (f  d \mu)\Vert_{L^{q} (\sigma)} \le \varkappa \, \Vert f \Vert_{L^p(\mu)}, \quad \text{for all} \,\,  f \in L^p(\mu), 
		\end{equation}
	 where characterized by E. T. Sawyer \cite{S} in the case $p=q>1$, R. L. Wheeden \cite{W} in the 
	 case $q>p>1$, and the second author \cite{V1}  in the case $0<q<p$ and $p>1$, 		along with their weak-type counterparts, 
		 \begin{equation}
		\label{two-weight-weak}
			\Vert M_\alpha (f  d \mu)\Vert_{L^{q, \infty} (\sigma)} \le \varkappa_w \, \Vert f \Vert_{L^p(\mu)}, \quad \text{for all} \,\,  f \in L^p(\mu),  
\end{equation}
where $\sigma, \mu  \in  \mathcal{M}^+(\R^n)$, and 
$\varkappa, \varkappa_w$ are positive constants which do not  dependent on $f$. 

 However, some of the methods 
		 used in \cite{V1} for $0<q<p$ and $p>1$ are not directly applicable in the case $p=1$, although there are analogues of these results for real Hardy spaces, i.e., when the norm 
		$ \Vert f \Vert_{L^p(\mu)}$ on the right-hand side of \eqref{two-weight} or \eqref{two-weight-weak}  is replaced with $ \Vert M_\mu f \Vert_{L^p(\mu)}$, where  
		 \begin{equation}
		\label{mu-max}
 M_\mu f (x) \defeq \sup_{Q \ni x} \frac{1}{|Q|_\mu} \int_Q |f| d \mu. 
 	\end{equation}

	We would like to understand similar problems in the case $0<q<1$ and $p=1$, in particular, 
	when $M_\alpha\colon \mathcal{M}^+(\R^n) \to L^q(d \sigma)$, or equivalently,  
	there exists a constant $\varkappa>0$ such that the inequality
		\begin{equation}
		\label{main_frac_inequality}
			\Vert M_\alpha \nu \Vert_{L^q(\sigma)} \le \varkappa \, \Vert\nu\Vert  
\end{equation}
		holds for all $\nu \in \mathcal{M}^+(\R^n)$.
		
	In the case $\alpha = 0$, Rozin \cite{R} showed that the condition 
		\[ \sigma \in L^{\frac{1}{1-q}, 1} (\R^n, dx)\]
	is sufficient for the Hardy-Littlewood operator $M=M_0\colon L^1(dx) \to L^q(\sigma)$ to be bounded; moreover, 
	when $\sigma$ is radially symmetric and  decreasing, this is also a necessary condition.
	We will generalize this result and provide necessary and sufficient conditions for the range $0 \le \alpha < n$. We also obtain analogous results 
	for the weak-type inequality 
	\begin{equation}
		\label{weak_inequality}
			\Vert M_\alpha \nu \Vert_{L^{q, \infty} (\sigma)} \le \varkappa_w \, \Vert\nu\Vert, \quad \text{for all} \,\,   \nu \in \mathcal{M}^+(\R^n), 
		\end{equation}
		 where $\varkappa_w$ is a positive constant which does not depend on $\nu$. 
		 
	We  treat more general maximal operators as well, in particular, dyadic maximal operators 
	\begin{equation}
		\label{dyadic-max}
			M_\rho \nu (x) \defeq  \sup_{Q \in \mathcal{Q}\colon  Q \ni x} \rho_Q \, {|Q|_\nu}, 
		\end{equation}
where $\mathcal{Q}$ is the family of all dyadic cubes in $\R^n$, and $\{\rho_Q \}_{Q \in \mathcal{Q}}$ is a fixed sequence of nonnegative reals associated 
with $Q \in  \mathcal{Q}$. The corresponding  weak-type maximal inequality is given by 
	\begin{equation}
		\label{dyadic-ineq}
			\Vert M_\rho \nu  \Vert_{L^{q, \infty} (\sigma)} \le \varkappa_w \, \Vert\nu \Vert, \quad  \text{for all} \,\,    \nu\in \mathcal{M}^+(\R^n).  
		\end{equation}

\subsection{Strong-Type Inequality}
	\begin{theorem} 	Let $\sigma \in M^{+}(\R^n)$, $0<q<1$, and $0\le \alpha <n$.  
		The inequality (\ref{main_frac_inequality}) holds if and only if there exists a function $u\not\equiv 0$ such that  
		$$u \in L^q(\sigma), \quad \text{{\rm and}} \quad  u \ge M_\alpha(u^q \sigma).$$ 
		
		Moreover, $u$ can be constructed as follows: $u = \lim_{j \to \infty} u_j$, where $u_0 \defeq (M_\alpha \sigma)^{\frac{1}{1-q}}$, $u_{j+1} \ge u_j$, and 
		\begin{equation}
		\label{u-j}
 u_{j+1} \defeq M_\alpha (u_j^q \sigma), \quad j=0, 1, \ldots . 
 		\end{equation}
		In particular, $u \ge (M_\alpha \sigma)^{\frac{1}{1-q}}$. 
	\end{theorem}
	\begin{proof}
	($\Rightarrow$) We let $u_0 \defeq (M_\alpha \sigma)^{\frac{1}{1-q}}$. Notice that, for all $x \in Q$,  we have $u_0 (x) \ge \Big(\frac{|Q|_\sigma}{|Q|^{1-\frac{\alpha}{n}}}\Big)^{\frac{1}{1-q}}$. 
	Hence, 
	$$
	u_1(x) \defeq M_\alpha(u_0^q d \sigma)(x) =  \sup_{Q \ni x} \frac{1}{|Q|^{1-\frac{\alpha}{n}}}\int_Q u_0^q d \sigma \ge  \sup_{Q \ni x} \Big(\frac{|Q|_\sigma}{|Q|^{1-\frac{\alpha}{n}}}\Big)^{\frac{1}{1-q}}=u_0(x).
	$$
	By induction, we see that 
	$$u_{j+1}\defeq M_\alpha(u_j^q d \sigma) \ge M_\alpha(u_{j-1}^q d \sigma) = u_j, \quad j=1, 2, \ldots .$$

Let $u = \lim u_j$. By (\ref{main_frac_inequality}), we have
		\begin{align*}
			\Vert u_{j+1} \Vert_{L^q(\sigma)}
				&= \Vert M_\alpha (u^q_j\sigma) \Vert_{L^q(\sigma)} \\
				&\le \varkappa \Vert u_j \Vert_{L^q(\sigma)}^q \\
				&\le \varkappa \Vert u_{j+1} \Vert_{L^q(\sigma)}^q.
		\end{align*}
	From this  we deduce that $\Vert u_{j+1} \Vert_{L^q(\sigma)}^q \le \varkappa^{\frac{1}{1-q}}$ for $j = 0, 1, \ldots$.
	Since the norms $\Vert u_j \Vert_{L^q(\sigma)}^q$ are uniformly bounded, by the Monotone Convergence Theorem, we have for $u \defeq \lim_{j \rightarrow \infty} u_j$ that $u \in L^q(\sigma)$.
	Note that by construction $u = M_\alpha (u^q d \sigma)$.
	
	($\Leftarrow$) We can assume here that $M_\alpha \nu$ is defined, for $\nu \in \mathcal{M}(\R^n)$, 
	as   the centered 
					fractional maximal function,
	$$M_\alpha \nu (x) \defeq \sup_{r>0} \frac{\nu(B(x, r))}{|B(x, r)|^{1-\frac{\alpha}{n}}}, $$
	 since it is equivalent to its uncentered analogue used above.  Suppose that there exists $u \in L^q(\sigma)$ ($u \not\equiv 0$) such that $u \ge M_\alpha(u^q d \sigma)$.
	Set $\omega \defeq u^q d \sigma$. Let $\nu \in \mathcal{M}(\R^n)$. 
	
	We note that we have
	\begin{align*}
		\frac{M_\alpha \nu(x)}{M_\alpha \omega(x)}
			&= \frac{ \sup_{r>0} \frac{|B(x, r)|_\nu}{|B(x, r)|^{1-\frac{\alpha}{n}}}}{ \sup_{\rho>0} \frac{|B(x, \rho)|_\nu}{|B(x, \rho)|^{1-\frac{\alpha}{n}}}} \\
			&\le \sup_{r>0} \frac{|B(x, r)|_\nu}{|B(x, r)|_\omega} \\
			&=: M_\omega \nu(x).
	\end{align*}
	Thus,
	\begin{align*}
		\Vert M_\alpha \nu \Vert_{L^q(\sigma)}
			&= \left\Vert \frac{M_\alpha \nu}{M_\alpha \omega} \right\Vert_{L^q((M_\alpha \omega)^q d \sigma)} \\
			&\le \left\Vert \frac{M_\alpha \nu}{M_\alpha \omega} \right\Vert_{L^q(d\omega)} \\
			&\le \left\Vert M_\omega \nu  \right\Vert_{L^q(d\omega)}\\
						&\le C \left\Vert M_\omega \nu \right\Vert_{L^{1,\infty}(\omega)}\le C \Vert \nu \Vert,
	\end{align*}
	by Jensen's inequality and the $(1, 1)$-weak-type  maximal function inequality for $M_\sigma \nu$. 
	This establishes (\ref{main_frac_inequality}).
\end{proof}

\subsection{Weak-Type Inequality}

For $0\le \alpha<n$, we define the \textit{Hausdorff content} on a set $E \subset \R^n$ to be
		\begin{equation}
			H^{n-\alpha}(E) \defeq \inf \left\{ \sum_{i=1}^\infty r_i^{n-\alpha}\colon E \subset \bigcup_{i=1}^\infty B(x_i, r_i),  \right\}
		\end{equation}
		where the collection of balls $\{B(x_i, r_i)\}$ forms a countable covering of $E$.
	\begin{theorem}
		Let $\sigma \in M^{+}(\R^n)$, $0<q<1$, and $0\le \alpha <n$. Then the following  conditions are equivalent:
		\begin{enumerate}
			\item There exists a positive constant $\varkappa_w$ such that 
			$$\Vert M_\alpha \nu \Vert_{L^{q, \infty}(\sigma)} \le \varkappa_w \, \Vert\nu\Vert 
			\quad 
			\text{{\rm for all}} \, \, \nu \in \mathcal{M}(\R^n).$$
			\item There exists a positive constant $C>0$ such that  
			$$\sigma(E) \le C \, (H^{n - \alpha}(E))^q \quad \text{{\rm for all Borel sets}} 
			 \, \, E \subset \R^n.$$
		\item $M_\alpha \sigma \in L^{\frac{q}{1-q}, \infty}(\sigma)$.
		\end{enumerate}
\end{theorem}

\begin{remark} In the case $\alpha=0$ each of the conditions {\rm (1)--(3)} is equivalent to $\sigma \in L^{\frac{q}{1-q}, \infty}(dx)$. 
\end{remark}

	\begin{proof}
		(1) $\Rightarrow$ (2) Let $K\subset E$ be a compact set in $\R^n$ such that $H^{n-\alpha}(K)>0$. 
It follows from  Frostman's theorem (see the proof of Theorem 5.1.12 in \cite{AH}) that there exists a measure $\nu$ supported on $K$ such that $\nu(K)\le H^{n-\alpha}(K)$, 
and, for every $x\in K$ there exists a cube $Q$ such that $x \in Q$ and $|Q|_\nu\ge c \, |Q|^{1-\frac{\alpha}{n}}$, where $c$ depends only on $n$ and $\alpha$.  		
		Hence,
		$$
		M_\alpha \nu(x) \ge \sup_{Q \ni x} \frac{|Q|_\nu}{|Q|^{1-\frac{\alpha}{n}}}\ge c \quad \text{for all} \, \, x \in K, 
		$$
		where $c$ depends only on $n$ and $\alpha$. Consequently, 
		$$
		c^q \, \sigma(K) \le \Vert M_\alpha \nu\Vert^q_{L^{q, \infty} (\sigma)}\le \varkappa_w^q 
		\Big(H^{n-\alpha}(K)\Big)^q.
		$$
		
		If $H^{n-\alpha}(E)=0$, then $H^{n-\alpha}(K)=0$ for every compact set $K\subset E$, and 
		consequently $\sigma(E)=0$. Otherwise, 
		$$\sigma(K)\le \varkappa_w^q  \Big(H^{n-\alpha}(K)\Big)^q\le \varkappa_w^q \Big(H^{n-\alpha}(K)\Big)^q,
		$$
		for every compact set $K\subset E$, which proves (2) with $C=c^{-q}\varkappa_w^q$.

		(2) $\Rightarrow$ (3)
			Let $E_t \defeq \{ x: M_\alpha \sigma (x) > t \}$, where $t>0$. Let $K \subset E_t$ be a compact set. 
			Then for each $x \in K$ there exists $Q_x \ni x$ such that
			\[ \frac{\sigma(Q_x)}{|Q_x|^{1-(\frac{\alpha}{n})}} > t. \]
			
			Now consider the collection $\{ Q_x \}_{x \in K}$, which forms a cover of $K$.
			By the Besicovitch covering lemma, we can find a  subcover $\{Q_i\}_{i \in I}$, where $I$ is a countable index set, such that $K \subset \bigcup_{i \in I} Q_i$ and $x \in K$ is contained in at most $b_n$ sets in $\{ Q_i\}$.
			By (2), we have
				\[ \sigma(K) \le [ H^{n-\alpha}(K) ]^q, \]
			and by the definition of the Hausdorff content we have
				\[  H^{n-\alpha}(K) \le \sum |Q_i|^{1-(\alpha / n)}.\]
			Since $\{Q_i\}$ have bounded overlap, we have
				\[ \sum_{i \in I} \sigma(Q_i) \le b_n \sigma(K). \]
			Thus, 
				\[ \sigma(K) \le \left( b_n \frac{\sigma(K)}{t} \right)^q, \]
			which shows that
				\[ t^{\frac{q}{1-q}} \sigma(K) \le (b_n)^{\frac{1}{1-q}} < + \infty. \]
			Taking the supremum over all $K \subset E_t$ in the preceding inequality, we deduce  $M_\alpha \sigma \in L^{\frac{q}{1-q}, \infty}(\sigma)$.

					(3) $\Rightarrow$ (1). We can assume again that $M_\alpha$ is the centered 
					fractional maximal function, since it is equivalent to the uncentered version. 
				Suppose that $M_\alpha \sigma \in L^{\frac{q}{1-q}, \infty}(\sigma)$. 
					Let $\nu \in \mathcal{M}(\R^n)$. 
					Then, as in the case of the strong-type inequality, 
				\begin{align*}
		\frac{M_\alpha \nu(x)}{M_\alpha \sigma(x)}
			&= \frac{ \sup_{r>0} \frac{|B(x, r)|_\nu}{|B(x, r)|^{1-\frac{\alpha}{n}}}}{ \sup_{\rho>0} \frac{|B(x, \rho)|_\sigma}{|B(x, \rho)|^{1-\frac{\alpha}{n}}}} \\
			&\le \sup_{r>0} \frac{ |B(x, r)|_\nu}{|B(x, r)|_\sigma}=: M_\sigma \nu(x).
	\end{align*}
	Thus, by H\"{o}lder's inequality for weak $L^p$-spaces,
	\begin{align*}
		\Vert M_\alpha \nu \Vert_{L^{q, \infty}(\sigma)}
			&\le \Vert (M_\alpha \sigma) \, (M_\sigma \nu) \Vert_{L^{q, \infty}(\sigma)} \\
			&\le \Vert M_\alpha \sigma\Vert_{L^{\frac{q}{1-q}, \infty}(\sigma)} \, 
			\Vert M_\sigma \nu \Vert_{L^{1, \infty}(\sigma)}  \\
			&\le c \Vert M_\alpha \sigma\Vert_{L^{\frac{q}{1-q}, \infty}(\sigma)} \,  \Vert \nu \Vert,
				\end{align*}
				where in the last line 
				we have used the $(1, 1)$-weak-type  maximal function inequality for the centered maximal function $M_\sigma \nu$. \end{proof}
				
				We now characterize weak-type $(1, q)$-inequalities \eqref{dyadic-ineq} 
								for the generalized dyadic maximal operator $M_\rho$ defined by \eqref{dyadic-max}. 
								The corresponding $(p, q)$-inequalities in the case $0<q<p$ and $p>1$ 
								were characterized in \cite{V1}. The results obtained in \cite{V1} for weak-type inequalities remain 
								valid in the case $p=1$, but some elements of the proofs must be modified as indicated 
								below. 
							\begin{theorem}
		Let $\sigma \in \mathcal{M}^{+}(\R^n)$, $0<q<1$, and $0\le \alpha <n$. Then the following  conditions are equivalent:
		\begin{enumerate}
			\item There exists a positive constant $\varkappa_w$ such that \eqref{dyadic-ineq}  holds. \smallskip 
			
						\item   
	$M_\rho \sigma \in L^{\frac{q}{1-q}, \infty}(\sigma)$.
		\end{enumerate}
\end{theorem}

\begin{proof} (2) $\Rightarrow$ (1) The proof of this implication is similar to the case of fractional maximal operators. Let $\nu \in \mathcal{M}(\R^n)$. Denoting by $Q, P \in \mathcal{Q}$ dyadic cubes in $\R^n$, we estimate 
				\begin{align*}
		\frac{M_\rho \nu(x)}{M_\rho \sigma(x)} 
			&  = \frac {\sup_{Q \ni x} (\rho_Q \, |Q|_\nu)}{\sup_{P \ni x} (\rho_P \, |P|_\sigma)}
		 \\
			&\le \sup_{Q \ni x} \frac {|Q|_\nu}{|Q|_\sigma}=: M_\sigma \nu(x).
	\end{align*}
	By H\"{o}lder's inequality for weak $L^p$-spaces,
	\begin{align*}
		\Vert M_\rho \nu \Vert_{L^{q, \infty}(\sigma)}
			&\le \Vert (M_\rho \sigma) \, (M_\sigma \nu) \Vert_{L^{q, \infty}(\sigma)} \\
			&\le \Vert M_\rho \sigma\Vert_{L^{\frac{q}{1-q}, \infty}(\sigma)} \, 
			\Vert M_\sigma \nu \Vert_{L^{1, \infty}(\sigma)}  \\
			&\le c \Vert M_\rho \sigma\Vert_{L^{\frac{q}{1-q}, \infty}(\sigma)} \,  \Vert \nu \Vert,
				\end{align*}
				by the $(1, 1)$-weak-type  maximal function inequality for the dyadic maximal function $M_\sigma$. 
				
				(1) $\Rightarrow$ (2) We set $f = \sup_{Q} (\lambda_Q \chi_Q)$ and $d \nu = f \, d \sigma$, where 
				$\{\lambda_Q\}_{Q \in \mathcal{Q}}$ is a finite sequence of non-negative reals. Then obviously 
			$$
		 M_\rho \nu (x)\ge \sup_{Q} (\lambda_Q \rho_Q \chi_Q), \quad \text{and} \quad 
	  \Vert\nu\Vert \le \sum_Q \lambda_Q \, |Q|_\sigma.				
		$$
	By (1), for all $\{\lambda_Q\}_{Q \in \mathcal{Q}}$, 
		$$
		\Vert \sup_{Q} ( \lambda_Q \rho_Q \chi_Q) \Vert_{L^{q, \infty} (\sigma)}\le \varkappa_v \, 
		\sum_Q \lambda_Q \, |Q|_\sigma. 
		$$ 
		Hence, by Theorem 1.1 and Remark 1.2 in \cite{V1}, it follows that (2) holds. 

\end{proof}
	
		\section{Carleson Measures for Poisson Integrals}			 
		
		In this section we treat $(1, q)$-Carleson measure inequalities for Poisson integrals with respect to Carleson measures  $\sigma\in \mathcal{M}^{+}(\R^{n+1}_{+})$		in the upper half-space  $\R^{n+1}_{+} = (x, y)\colon x \in \R^n, y>0$. The corresponding weak-type $(p, q)$-inequalities for all $0<q<p$ as well as strong-type $(p, q)$-inequalities for $0<q<p$ and $p>1$, were characterized  	in \cite{Vid}. Here we consider strong-type inequalities of the type 
		\begin{equation}\label{carl}
		\Vert \mathbf{P} \nu\Vert_{L^q(\R^{n+1}_{+}, \sigma)}\le \varkappa \, \Vert \nu\Vert_{\mathcal{M}^{+}(\R^{n})}, 
		\quad \text{{\rm for all}} \, \, \nu \in \mathcal{M}^{+}(\R^n), 
		\end{equation}
		for some constant $\varkappa>0$, where $\mathbf{P}\nu$ is the Poisson 
		integral of  $\nu \in  \mathcal{M}^{+}(\R^{n})$ defined by 
		$$\mathbf{P} \nu(x, y) \defeq \int_{\R^n} P(x-t, y) d \nu(t), \quad (x, y) \in \R^{n+1}_{+}.$$
		Here $P(x, y)$ denotes the Poisson kernel associated with $\R^{n+1}_{+}$. 
		
		By $\mathbf{P}^{*} \mu$ we denote the formal adjoint (balayage) operator defined, for  $ \mu \in  \mathcal{M}^{+}(\R^{n+1}_{+})$, by 
		$$\mathbf{P}^{*} \mu(t) \defeq \int_{\R^{n+1}_{+}} P(x-t, y) d \mu (x, y), \quad t \in \R^n.$$ 
		We will also need the symmetrized potential defined, for $ \mu \in  \mathcal{M}^{+}(\R^{n+1}_{+}) $, by 
		$$\mathbf{P}\mathbf{P}^{*} \mu(x, y) \defeq \mathbf{P}\Big [(\mathbf{P}^{*} \mu)dt\Big] = \int_{\R^{n+1}_{+}} P(x-\tilde x, y+\tilde y) d \mu (\tilde x, \tilde y), \quad (x, y) \in \R^{n+1}_{+}.$$ 
		As we will demonstrate below, the kernel of $\mathbf{P}\mathbf{P}^{*} \mu$ satisfies the weak maximum 
		principle with constant $h=2^{n+1}$.

	\begin{theorem}\label{carl-thm}  	Let $\sigma\in \mathcal{M}^{+}(\R^{n+1}_{+})$, and let $0<q<1$.  Then 
		inequality (\ref{carl}) holds if and only if there exists a function $u> 0$ such that  
		$$u \in L^q(\R^{n+1}_{+}, \sigma), \quad \text{{\rm and}} \quad  u \ge \mathbf{P} \mathbf{P}^{*}(u^q \sigma) \quad {\rm in} \, \, \R^{n+1}_{+}.$$ 
		
		Moreover, if (\ref{carl}) holds, then a positive solution $u= \mathbf{P} \mathbf{P}^{*}(u^q \sigma)$ 
		such that $u \in L^q(\R^{n+1}_{+}, \sigma)$ 
		can be constructed as follows: $u = \lim_{j \to \infty} u_j$,    where  
		\begin{equation}
		\label{u-j-j}
 u_{j+1} \defeq \mathbf{P} \mathbf{P}^{*} (u_j^q \sigma), \quad j=0, 1, \ldots , \quad  u_0 \defeq c_0 (\mathbf{P} \mathbf{P}^{*} \sigma)^{\frac{1}{1-q}}, 
 		\end{equation}
		for a small enough constant $c_0>0$ (depending only on $q$ and $n$),  which ensures that $u_{j+1} \ge u_j$. 
		In particular, $u \ge c_0 \, (\mathbf{P} \mathbf{P}^{*} \sigma)^{\frac{1}{1-q}}$. 
	\end{theorem}
	\begin{proof} We first prove that \eqref{carl} holds if and only if 
	\begin{equation}\label{carl-1}
		\Vert \mathbf{P} \mathbf{P}^{*} \mu\Vert_{L^q(\R^{n+1}_{+}, \sigma)}\le \varkappa \, \Vert \mu\Vert_{\mathcal{M}^{+}(\R^{n+1}_{+})}, 
		\quad \text{{\rm for all}} \, \, \mu \in \mathcal{M}^{+}(\R^{n+1}_{+}). 
		\end{equation}
 Indeed, letting $\nu = \mathbf{P}^{*} \mu$ in \eqref{carl} yields 
\eqref{carl-1} with the same embedding constant $\varkappa$. 

Conversely, suppose that \eqref{carl-1} holds. Then by Maurey's factorization theorem (see \cite{M}), there exists 
$F \in  L^1(\R^{n+1}_{+}, \sigma)$ such that $F>0 \, \, d \sigma$-a.e., and 
\begin{equation}\label{maurey}
	\Vert F \Vert_{L^1(\R^{n+1}_{+}, \sigma)}\le 1, \quad
  	\sup_{(x, y) \in \R^{n+1}_{+}}  \mathbf{P} \mathbf{P}^{*} (F^{1-\frac{1}{q}} d \sigma) (x, y)\le \varkappa. 
		\end{equation}
		By letting $y \downarrow 0$ in \eqref{maurey} and using the Monotone Convergence Theorem, we deduce  
\begin{equation}\label{maurey-1}
		\sup_{x \in \R^{n}}  \mathbf{P}^{*} (F^{1-\frac{1}{q}} d \sigma) (x)\le \varkappa. 
		\end{equation}
	Hence, by Jensen's inequality  and \eqref{maurey-1}, for any $\nu \in  \mathcal{M}^{+}(\R^n)$, we have 
	$$
	\Vert \mathbf{P} \nu\Vert_{L^q(\R^{n+1}_{+}, \sigma)}\le \Vert \mathbf{P} \nu\Vert_{L^1(\R^{n+1}_{+}, \, F^{1-\frac{1}{q}} d \sigma)} = \Vert \mathbf{P}^{*} (F^{1-\frac{1}{q}} d \sigma)\Vert_{L^1(\R^{n}, \, d \nu)}
	\le \varkappa \, \Vert \nu\Vert_{\mathcal{M}^{+}(\R^{n})}. 
	$$
	
	We next show that the kernel of $\mathbf{P} \mathbf{P}^{*}$ satisfies the weak maximum principle with constant 
	$h=2^{n+1}$.   Indeed, suppose $\mu \in \mathcal{M}^{+}(\R^{n+1}_{+})$, and 
	$$
	 \mathbf{P} \mathbf{P}^{*} \mu(\tilde x, \tilde y) \le M, \quad \text{for all} \, \,  (\tilde x, \tilde y) \in S(\mu). 
	$$
	Without loss of generality we may assume that $S(\mu) \Subset \R^{n+1}_{+}$ is a compact set. 
	For $t \in \R^n$, let  
	 $(x_0, y_0)\in S(\mu)$ be a point such that 
	$$|(t, 0) - (x_0, y_0)|=\text{dist} \Big((t, 0), S(\mu)\Big).$$ 
	Then by the triangle inequality, for any $(\tilde x, \tilde y) \in S(\mu)$, 
	$$
	|(x_0, y_0) - (\tilde x, -\tilde y)|\le  | (x_0, y_0) - (t, 0) | + |(t, 0)- (\tilde x, -\tilde y)|\le 2 |(t, 0)- (\tilde x, \tilde y)|. 
	$$
	Hence, 
	$$
	\sqrt{|t-\tilde x|^2 + \tilde y^2} \ge \frac{1}{2} \, \sqrt{\Big [ |x_0-\tilde x|^2 + (y_0 + \tilde y)^2\Big]}. 
	$$
	It follows that, for all $t \in \R^n$ and $(\tilde x, \tilde y) \in S(\mu)$, we have  
	$$
	P(t-\tilde x, \tilde y)\le 2^{n+1} P(x_0-\tilde x, y_0 +\tilde y). 
	$$
	Consequently,  for all $t \in \R^n$, 
	$$
	\mathbf{P}^{*} \mu(t) \le 2^{n+1} \mathbf{P} \mathbf{P}^{*} \mu(x_0, y_0)\le  2^{n+1} M. 
	$$
	Applying the Poisson integral $\mathbf{P}[dt]$ to both sides of the preceding inequality, we obtain 
	$$
 \mathbf{P} \mathbf{P}^{*} \mu(x, y) \le 2^{n+1} M \quad \text{for all} \,\, (x, y) \in   \R^{n+1}_{+}. 
	$$ 
	
	This proves that the weak maximum principle  holds for $\mathbf{P} \mathbf{P}^{*}$ 
	with $h=2^{n+1}$. It follows from Theorem~\ref{strong-thm} that \eqref{carl} holds if and only if there exists a non-trivial 
	$u \in L^q(\R^{n+1}_{+}, \sigma)$ such that $u \ge \mathbf{P} \mathbf{P}^{*}(u^q d\sigma)$. 
	Moreover, a positive solution $u= \mathbf{P} \mathbf{P}^{*}(u^q \sigma) $ can be constructed as in the statement of Theorem~\ref{carl-thm} (see details in \cite{QV}).  
	\end{proof}
	
	\begin{cor}\label{carl-cor}  Under the assumptions of Theorem~\ref{carl-thm}, inequality 	 (\ref{carl}) holds if and only if there exists a function $\phi \in L^1(\R^n)$, $\phi>0$ a.e., such that  
		$$  \phi \ge \mathbf{P}^{*} \Big [(\mathbf{P}\phi)^q d \sigma\Big] \quad {\rm a.e. \, \, in} \, \, \R^{n}.$$ 
		Moreover, if 	 (\ref{carl}) holds, then there exists a positive solution $\phi\in L^1(\R^n)$ to the equation 
		$  \phi = \mathbf{P}^{*} \Big [(\mathbf{P}\phi)^q d \sigma\Big]$.  
	\end{cor}
	
	\begin{proof} If 	 (\ref{carl}) holds then by Theorem~\ref{carl-thm} there exists $u=\mathbf{P} \mathbf{P}^{*}(u^q d \sigma)$ such that $u>0$ and $u \in L^q(\R^{n+1}_{+}, \sigma)$. Setting $\phi = \mathbf{P}^{*} (u^q d \sigma)$, we see that 
	$$\mathbf{P}  \phi = \mathbf{P}\mathbf{P}^{*} (u^q d \sigma) = u,$$
	so that $\phi=\mathbf{P}^{*} [(\mathbf{P} \phi)^q d \sigma]$, and consequently  
	$$\Vert \phi\Vert_{L^1(\R^n)} = \Vert u\Vert^q_{L^q(\R^{n+1}_{+}, \sigma)}=\int_{\R^n} u(x, y) dx <\infty.$$

	Conversely, if there exists $\phi>0$, $\phi \in L^1(\R^n)$ such that  $ \phi \ge \mathbf{P}^{*} \Big [(\mathbf{P}\phi)^q d \sigma\Big]$, then letting $u=\mathbf{P}\phi$, we 
	see that $u$ is a positive harmonic function in $\R^{n+1}_{+}$ so that  
	 $$u = \mathbf{P}\phi \ge \mathbf{P}\mathbf{P}^{*} (u^q d \sigma),$$
	and for all $y>0$, 
	 $$
	 \Vert u\Vert^q_{L^q(\R^{n+1}_{+}, \sigma)} = \int_{\R^n} \Big[\mathbf{P}\mathbf{P}^{*} (u^q d \sigma)\Big] (x, y) \, dx \le \int_{\R^n} u(x, y) dx =\Vert \phi\Vert_{L^1(\R^n)}<\infty.
	$$ 
	Hence, inequality (\ref{carl}) holds  by Theorem~\ref{carl-thm}.  
	\end{proof}

%%%%%%%%%%%%%%%%%%%%%%%%%%%%%%%%%%%%%%%%%%%%%%%%
%%%%%%%%%%%% Bibliography %%%%%%%%%%%%%%%%%%%%%%
%%%%%%%%%%%%%%%%%%%%%%%%%%%%%%%%%%%%%%%%%%%%%%%%

\end{document}